\documentclass[11pt]{article}
\usepackage{amsmath, amscd, amsfonts, amssymb, amsthm, tikz, times}
\usepackage[T1]{fontenc}
\newtheorem{thm}{Theorem}[section]

\newtheorem{lem}[thm]{Lemma}
\newtheorem{prop}[thm]{Proposition}
\theoremstyle{definition}
\newtheorem{ex}[thm]{Example}
\newtheorem{defi}[thm]{Definition}

\def\sas{{S$\aa$S}}
\def\fdd{\stackrel{fdd}{\Longrightarrow}}
\def\ed{\stackrel{d}{=}}
\def\bi{\begin{itemize}}
\def\ei{\end{itemize}}
\def\aa{\alpha}

\def\eps{\epsilon}
\def\th{\theta}
\def\om{\omega}
\def\Om{\Omega}

\def\be{\begin{equation}}
\def\ee{\end{equation}}
\def\bea{\begin{eqnarray}}
\def\eea{\end{eqnarray}}
\def\nn{\nonumber}
\def\ff{\infty}
\def\({\left(}
\def\){\right)}
\def\[{\left[}
\def\]{\right]}
\def\lk{\left\{}
\def\rk{\right\}}
\def\lb{\left|}
\def\rb{\right|}

\def\dist{\text{dist}}
\def\t{\tau}
\def\tH{\tau^{H'}}

\def\H{\tilde H}

\def\g{g}

\def\R{\mathbb{R}}
\def\Z{\mathbb{Z}}

\def \P{\mathbf{P}}
\def \E{\mathbf{E}}

\def \BB{{\cal B}}

\def \FF{{\cal F}}
\def \GG{{\cal G}}

\begin{document}

\title{Random-time isotropic fractional stable fields}
\author{Paul Jung\footnote{Research was started at Sogang University and supported in part by Sogang University research grant 201010073}
\\[2ex]
{\normalsize Department of Mathematics, University of Alabama Birmingham}\\[1ex]
   }
\date{December 21, 2011} \maketitle
\abstract{Generalizing both Substable FSMs and Indicator FSMs, we introduce $\aa$-stabilized subordination, a procedure which produces new  FSMs ($H$-sssi \sas\ processes) from old ones. We extend these
processes to isotropic stable fields which have stationary increments in the
strong sense, i.e., processes which are invariant under Euclidean rigid
motions of the multi-dimensional time parameter. We also prove a Stable Central Limit Theorem which provides an intuitive picture of 
$\aa$-stabilized subordination.  Finally we show that $\aa$-stabilized subordination of Linear FSMs produces null-conservative FSMs,
a class of FSMs introduced in Samorodnitsky (2005).
}
\vspace{5mm}

\\

\tableofcontents
\newpage
\section{Introduction}
It is well-known that, up to constant multiples, the one-parameter family of standard fractional Brownian motions
$$\{(B_H(t))_{t\ge 0}\}_{H\in(0,1]}$$ are the only $H$-self-similar, stationary increment ($H$-sssi) Gaussian processes.
The term ``standard'' means that the variance at time $t=1$ is equal
to one.  The parameter $H$ is called the Hurst parameter, and it is
also referred to as the self-similarity exponent since \be\label{ss}
(B_H(ct))_{t\ge 0}\ed(c^HB_H(t))_{t\in \R}. \ee

Isotropic fractional Brownian fields are Gaussian fields $\(B_H(t)\)_{t\in\R^n}$ such that $B_H(0)=0$ and $$\E(B_H(t)-B_H(s))^2=\|t-s\|^{2H}.$$
We have used the term {\it isotropic} to distinguish the fractional Brownian fields of \cite{lindstrom1993fractional}, which we discuss
in the present work, from the
fractional Brownian fields of \cite{dobric2006fractional} which are of a different nature.

The notion of isotropic self-similarity for random fields is the same as that for processes, i.e., the condition given in \eqref{ss} with
$t\in\R$ replaced by $t\in\R^N$.
Note that there is also a notion of anisotropic self-similar fields which allows for different scalings in different directions (cf. \cite{xiao2011properties}),
however, we will consider only isotropic fields.

For isotropic random fields, the natural notion of stationary increments is what is known as
 {\it stationary increments in the strong sense} (sis). Let
$\GG(\R^N)$ be the set of Euclidian rigid motions in $\R^N$. We say
$(X_t)_{t\in\R^N}$ is {\it sis} if for all $\g\in\GG(\R^N)$,
\be\label{def:sis} \(X_{\g(t)}-X_{\g(0)}\)_{t\in\R^N} \ed
\(X_t-X_0\)_{t\in\R^N} \ee where $\ed$ denotes equality of the
finite-dimensional distributions. In other words, the
finite-dimensional distributions of $(X_t)_{t\in\R^N}$ are invariant
under Euclidean rigid motions. It is not hard to see using the
covariance characterization of Gaussian fields that, up to constant
multiples, fractional Brownian fields are the only $H$-sssis
Gaussian fields.

For $0<\aa<2$, consider the (jointly) measurable symmetric $\aa$-stable (\sas) generalization of isotropic fractional Brownian fields.
Then unlike the Gaussian case, for each  $0<\aa<2$ and each $H\in(0,\max\{1,1/\aa\})$\footnote{FSFs with $H=\max\{1,1/\aa\}$ are also possible, but other than $\aa=1$ they are unique
\cite{samorodnitsky19901, CS}.},
there are myriads of {\bf $H$-sssis \sas\ fields} which are called $(\aa,H)$-processes in \cite{takenaka1991integral}.
Using more updated terminology, we shall call such processes {\bf (isotropic) fractional stable fields} (FSFs or $(\aa,H)$-FSFs
when parameters $(\aa,H)$ need to be specified),
see for example \cite{kono1991self,samorodnitsky1994stable,kokoszka1994new,pipiras2002structure,pipiras2004dilated,CS,xiao2011properties}
and the references therein.
When $t\in\R$, we call FSFs, simply, fractional stable motions (FSMs). Note that while
Gaussian fields are easily constructed and characterized by their covariances,
there is no direct generalization\footnote{For stable processes, there are notions similar to covariances called covariations and codifferences, see 
\cite[Chapter 2]{samorodnitsky1994stable}, and there is also a notion of spectral measure \cite{kuelbs1973representation}, but none
of these is perfectly analogous to the beautiful characterization of Gaussian processes in terms of positive definition functions.} of this convenient characterization to \sas\ fields.  Thus, classes of fractional stable fields must
be constructed one by one.

Partly motivated by a financial model discussed in \cite{bouchaud2004fluctuations}, we introduce {\it $\aa$-stabilized subordination}
which we use to construct several FSFs
by combining mechanisms which create both positive and negative dependence.
The procedure generalizes an argument of \cite{jung2010indicator} and produces what we call {\it $\aa$-stabilized random-time kernels}.
Given an integral kernel representation for an FSF, we introduce a random element into the kernel by (i) non-monotonic subordination and,
if needed, (ii) expansion of the stable random measure to include the new source of randomness (we refer to the second part as $\aa$-stabilization; see Theorem \ref{thm:normalization} below). 
The random-time kernel provides an integral representation for an FSF with a different Hurst parameter than that of
the original FSF.
As in typical time-subordinations, our $\aa$-stabilized subordinations replace the time index of an FSF,
$X_s$, with another process, $\t_t$. However, unlike typical
time-subordinations, we do not require $\t_t$ to be an increasing
Levy process, but rather to be an $H$-sssis vector-valued field (thus we will say the process is {\it non-monotonically
subordinated}); in fact,
since $X_s$ is a field, the range-space of $\t_t$ is not even
required to be linearly ordered.

FSFs obtained using $\aa$-stabilized random-time  kernels comprise a broad range of FSFs; in fact some have been previously seen in the literature.
We will see in Example \ref{lt example} that  random-time kernels associated to $\aa$-stable
Levy motions give us indicator fractional stable motions \cite{jung2010indicator}.
Another case where random-time kernels give an alternative view on previously known FSFs is given
in Example \ref{slope example} where it is shown that substable FSFs
(also called subordinated FSFs, cf. \cite[Sec. 7.9]{samorodnitsky1994stable}) are given by
$\aa$-stabilized subordination of random-slope FSMs.

The rest of the paper is organized as follows.  In Section \ref{sec:rtt} we define $\aa$-stabilized random-time kernels,
which produce random-time FSMs and FSFs, and discuss their properties. 
In Section \ref{sec:examples} we give
some examples of random-time FSFs.
In Section \ref{sec:erg}, we use decompositions of stationary \sas\ processes, introduced in
\cite{rosinski1995structure, samorodnitsky2005null}, to analyze
random-time Linear FSMs (L-FSMs).  In particular, we show that random-time L-FSMs are in the class of
{\it null-conservative} FSMs. In the final section we discuss some open problems.

\section{$\aa$-stabilized random-time kernels}\label{sec:rtt}
\subsection{Definition}
Throughout the sequel, unless otherwise stated, we fix $0<\aa<2$. Integral representations of
measurable  $\aa$-stable processes, of the type
\begin{equation} \label{e:stab.process}
X_s = \int_{E} f_s(x) \; M_\aa(dx), \; s\in S,
\end{equation}
where $M_\aa$ is a \sas\ random measure on $(E,\BB)$ with
$\sigma$-finite control measure $m(dx)$, are well-known (see Chapters 11 and
13 of \cite{samorodnitsky1994stable}). The family $\{f_s\}_{s\in S}$
is a subset of $L^\aa(E,m)$ and is called a {\it (spectral)
representation} of $\(X_s\)_{s\in S}$. 

Henceforth we will typically let $t\in\R^N$ and let $s\in\R^d$.
Vector-valued processes $(\t_1(t),\ldots,\t_d(t))$ indexed by $t\in\R^N$ and taking values in $\R^d$ are the
so-called $(N,d)$-fields \cite[Sec. 6]{geman1980occupation}.
If an
$(N,d)$-field  is also $H$-sssis, then it is an $\R^d$-valued isotropic random vector field and we
will call it an $(N,d,H)$-field. The notions of self-similarity and stationary increments in the strong sense are the same for 
 random fields and random vector fields. Namely, the finite-dimensional joint distributions are invariant under
isotropic scaling and their increments are invariant under Euclidean rigid motions.

The following definition is an extension of the
procedure used to define {\it
indicator fractional stable motions} \cite{
jung2010indicator}.  As can be seen in Theorem \ref{thm:normalization}, it can be thought of as an $\aa$-stabilization of {\it iterated processes} or {\it processes at
random times}  (cf. \cite{burdzy1992some, nane2006laws, JM}).

\begin{defi}\label{rtdef}
Let $\{f_s\}_{s\in \R^d}\subset L^\aa(E,m)$ be a representation of an $(\aa,H)$-FSF, $X_s$, supported on $(\Om,\FF,\P)$ 
and let $\(\t_t\)_{t\in \R^N}$ be an $(N,d,H')$-field, with \be\E\|\t_t\|^{H\aa}<\ff,\ee supported on a different probability space
$(\Om',\FF',\P')$. The {\bf $\aa$-stabilized subordination} of $X_s$ with respect to $\t_t$ is given by
\be
X^\t_t:=\int_{E\times\Om'} f_{\t_t(\omega')}(x) \,M_\aa(dx \times d\omega'),
\ee
and it is represented by the {\bf $\aa$-stabilized random-time kernel}
\be \label{lpcondition}
\{f_{\t_t(\om')}\}_{t\in \R^N}\subset L^\aa(\Om'\times
E,\P'\times m).
\ee
\end{defi}

\noindent {Remarks:}
\begin{enumerate}
\item  FSFs which are produced using $\aa$-stabilized subordination will be called {\it random-time FSFs}. Although representations of $\aa$-stable processes
are not unique, let us see that for a given $X_s$ and $\t_t$, the process $X^\t_t$ is unique (in terms of finite-dimensional distributions).

Suppose $\{f_s\}$ on $(E,m)$ and $\{g_s\}$ on $(D,\pi)$ are two different representations of $X_s$. Fix a vector of times $(s_1,\ldots,s_n)$.  By Eq. 3.2.2 in \cite{samorodnitsky1994stable}, the characteristic function of the $n$-dimensional distribution $(X_{s_1},\ldots,X_{s_n})$ is given by
\bea
\phi(\theta_1,\ldots,\theta_n)&=& \exp\lk-\int_E|\sum_{j=1}^n \th_jf_{s_j}|^\aa \, m(dx)\rk\nn\\
&=& \exp\lk-\int_B|\sum_{j=1}^n \th_jg_{s_j}|^\aa \, \pi(dx)\rk\nn.
\eea
Therefore the characteristic function of $(X_{s_1}^\t,\ldots,X_{s_n}^\t)$ is given by
\be
\exp\lk-\int_E\E'|\sum_{j=1}^n \th_jf_{\t_{s_j}}|^\aa \, m(dx)\rk\\
= \exp\lk-\int_B\E'|\sum_{j=1}^n \th_jg_{\t_{t_j}}|^\aa \, \pi(dx)\rk.
\ee
\item
Condition \eqref{lpcondition} is needed for $\{f_{\t_t(\om)}\}_{t\in \R^N}$ to be a well-defined representation
(see Section 3.2 of
\cite{samorodnitsky1994stable} for details).
To see that \eqref{lpcondition} indeed holds, using self-similarity (see Lemma \ref{lemma: ss} below) we see that
\bea\label{welldefined}
\nn\int_{\Om'}\int_{E} |f_{\t_t(\omega')}(x)|^\aa \, dx \,\P'(d\omega') &=& \E'\int_{E} \|\t_t(\omega')\|^{H\aa}|f_{1}(x)|^\aa \, dx\\
&=& C\E'\|\t_t\|^{H\aa}<\infty\eea 
where $C=\int_{\R} |f_{1}(x)|^\aa \, dx$.
\item The above definition and the following results can be extended to vector-valued FSFs with representations involving vectors
of $L^\aa$ functions. However, the notation becomes cumbersome, so we have restricted ourselves to real-valued FSFs
and vector-valued subordinators.
\end{enumerate}

\subsection{Motivations}

Before delving into equations and proofs, let us discuss some motivations behind Definition \ref{def:sis}.

From a purely mathematical standpoint, we have the following two motivations for $\aa$-stabilized random-time kernels:

\bi

\item Definition \ref{rtdef} provides quite general conditions under which
$H$-sssi non-monotonic subordination of kernels produces FSFs. In particular,
Theorem \ref{thm:sssi} expands upon the idea behind Indicator FSMs, an application of Definition \ref{rtdef} to a very specific kernel (see Example \ref{lt example}),
to show that the two properties of self-similarity and stationary increments are ``preserved'' under time-subordination.

\item We say that an FSM is dissipative, null-conservative, or positive conservative if its increment process is dissipative, null-conservative, or positive conservative (see Section \ref{sec:erg}).
The class of dissipative FSMs are well understood and their increments have canonical representations as mixed moving averages
\cite{surgailis1993stable, pipiras2002decomposition, pipiras2002structure}; additionally, \cite{rosinski2000decomposition}
extends this class to FSFs. 
The class of positive-conservative FSMs also have canonical representations \cite[Remark 2.6]{samorodnitsky2005null}
for their increment processes,
and a further decomposition of the increments into their harmonizable, cyclic, and non-periodic components is known \cite{pipiras2004stable}.
There are no canonical representations  for the increments of null-conservative FSMs, and thus, more examples of null-conservative FSMs give us a better handle
on them. In Section \ref{sec:erg} we exhibit, for each $(\aa,H)$ pair, a family of null-conservative FSMs.
\ei

On the other hand, more than being a generalization of Indicator FSMs just for the sake of generalizing, 
 an application of $\aa$-stabilized random-time kernels can be used 
to model certain effects in the stock market. 

In \cite{bouchaud2004fluctuations}, it was argued that approximate diffusive behavior in financial markets,
i.e. linear growth in the quadratic variation of a given market process,
was a result of a mixture of both positive and negative correlations between increments of the process. Although the processes
discussed in the current work have infinite variance and thus undefined correlations, we were nevertheless 
 motivated to find self-similar, stationary increment processes exhibiting both positive and negative dependence in a natural way such that
 the positive and negative dependencies could be teased apart. As such, in the next subsection we will see that if one starts with a kernel representation of an FSM
whose increments are positively dependent, then $\aa$-stabilized subordination to a self-similar process
introduces countering negative dependencies.

A discrete analysis of some of the continuous-time processes discussed in this work, and their relation to results of \cite{bouchaud2004fluctuations} will
be fleshed out in a subsequent work. However, for readers seeking applied motivations, we give a brief description of a possible discrete model. 

Let us consider two different
types of players in the stock market:

\begin{itemize}
\item[T:] A trader (liquidity taker) that is buying or selling stock in a company based on information
\item[M:] A market maker (liquidity provider) who trades for {\it edge} and not on information
\end{itemize}
Market makers are traders who are contracted by exchanges to
continuously have both bids and offers on a given stock symbol. 
If a market maker trades {\it only} for ``edge'' (this is atypical,
but can be considered an extreme case), then the trades will tend to
be cancellative or alternating in nature. For example, if market
maker Mary believes the ``true price'' of stock XYZ to be \$100,
then she may show a bid of \$99 and an ask of \$101 (an edge of \$1
on either side).  If trader Tom comes along and buys stock XYZ for
\$101, then Mary might increase both the bid and ask, perhaps to a
bid of \$100 and an ask of \$102.  If market conditions do not change, this should increase her chances of
buying rather than selling stock on her next trade.

Here are three related reasons for such an increase.  First of all,
Mary may believe that Tom has more information about the company
and/or market then her, and thus Mary's ``true price'' needs to be
adjusted. Secondly, if Mary hedges by buying the stock for \$100,
then she will have made an edge-profit of \$1, essentially risk-free
since she will no longer have exposure to the stock. Finally, Mary should
increase her ask price since she wants to limit any {\it further} exposure
to the stock.

The above paragraph describes why Mary's trades might tend to
alternate between buys and sells. This behavior taken by itself
would contradict the ``stylized fact'' that stock market returns
have no (or very little) autocorrelations
\cite{campbell1997econometrics,cont2001empirical}. In
\cite{bouchaud2004fluctuations}, it was argued that sequences of
trades (and thus stock price changes) strike a balance between being
dominated by liquidity takers (T) and by market makers (M). The T's
generally cause positive correlations while M's cause negative
correlations; as mentioned above, the combination results in diffusive-like behavior.
Here, a trade being dominated by T should be interpreted as
the direction of the stock price change being influenced by market
information, and a trade being dominated by M should be interpreted
as the direction of the stock price change being influenced by
reasons described in the preceding paragraph.

A discrete model can now be described as follows. Let stock price changes be
identified with trades, and suppose a sequence of
 trades with alternating signs (for simplicity assume finite variance) is placed at a node on a
graph. Other sequences of alternating trades are placed at other
nodes, and nearby nodes of distance $d$ have initial values which
are, on average, positively correlated. The average correlation goes
to zero as $d\to\infty$. Now consider a marker which moves from node
to node indicating the position of the last trade made. 

Think now of the marker as a random walk and think of the initial
trades at different nodes as random sceneries on those nodes.  The sign of the trades at a given node typically alternates between
successive visits of the random walk. If the trades (sceneries) at different
nodes are independent and the graph formed by the nodes is $\Z$, then normalized sums of such processes is precisely
the model considered in \cite{JM}. In the infinite variance case, such a model scales to an Indicator FSM.


\subsection{Properties}

Throughout the sequel we will assume the following \textit{Usual Conditions}:
\bea\label{uc}
\begin{cases}(X_s)_{s\in\R^d}&\text{ is a measurable }(\aa,H)\text{-FSF supported on }(\Om,\FF,\P)\\
\{f_s\}_{s\in\R^d}&\text{ is a representation of }X_s\\
(\t_t)_{t\in\R^N} &\text{ is a measurable }(N,d,H')\text{-field supported on }(\Om',\FF',\P'),\\
 &\text{different from }(\Om,\FF,\P),\text{ and it satisfies }\E'\|\t_t\|^{H\aa}<\ff\\

\end{cases} 
\eea

The following Stable Limit Theorem shows why we say the random-time kernels are ``$\aa$-stabilized''.

\begin{thm}\label{thm:normalization}
Assume the usual conditions in \eqref{uc}. If $\{X_s^{(l)}\}$ and $\{\t_t^{(l)}\}$ are i.i.d. copies of $X_s$ and $\t_t$, respectively, then
\be
 n^{-\frac{1}{\aa}}\sum_{l=1}^n X_{\t_t^{(l)}}^{(l)} \fdd X^\t_t.
\ee
\end{thm}

\begin{proof}
Fix $(\theta_1,\cdots,\theta_k)\in \R^k, (t_1,\cdots,t_k)\in \R^{Nk}$. Let $\{f_s\}_{s\in\R^d}$ be a representation of 
$X_s$ with control measure $m$ on $E$, and
let $$\Gamma_n(t):=n^{-\frac{1}{\aa}}\sum_{l=1}^n X_{\t_t^{(l)}}^{(l)}.$$
We compute the characteristic functions 
\begin{eqnarray}
\E\left[\exp\left( i\sum_{j=1}^k \theta_j\Gamma_n(t_j)\right)\right]&=&\E\left[\exp\left( in^{-\frac{1}{\aa}}\sum_{j=1}^k \theta_j X_{\t_{t_j}} \right)\right]^n \nonumber\\
&=&\E'\left[\exp\(-n^{-1}\int_E |\sum_{j=1}^n\th_j f_{\t_{t_j}}|^\aa m(dx)\)\right]^n.
\end{eqnarray}
It is enough now to show that
\begin{equation}\label{eq2.2}
\E'\left[\exp\(-n^{-1}Z\)\right]=1-n^{-1}\E'Z+o(n^{-1})
\end{equation}
where $Z(\omega'):=\int_E |\sum_{j=1}^n\th_j f_{\t_{t_j}}|^\aa m(dx)$.
Note that $\E'Z<\ff$ by \eqref{welldefined}. 

Eq. \eqref{eq2.2} is proved by
$$n\left(1-\E'\left[\exp(-n^{-1} Z)\right]\right)\mathop{\longrightarrow}_{n\rightarrow\infty} \E'Z,$$
which follows from the Dominated Convergence Theorem: 
$n(1-\exp(-n^{-1}Z))$ converges almost surely to $Z$ and is almost surely bounded from above by $Z$ which is integrable with respect to $\P'$.
\end{proof}

Let us now show that random-time FSFs are legitimate isotropic fractional stable fields. We first need the following lemma which is a slight generalization of Proposition
7.3.6 in \cite{samorodnitsky1994stable} to isotropic random fields:

\begin{lem}\label{lemma: ss}
Assume the usual conditions in \eqref{uc}. The family
$\{f_s\}_{s\in\R^d}$ is a representation of $X_s$ if and only if for
any $n\ge 1$ and $\theta_j\in\R$, $s_j\in\R^d$, $1\le j\le n$, the
following integral does not depend on $c>0$ nor on the Euclidean rigid motion $\g\in\GG$:

 \be c^{-H\aa}\int_E \lb\sum_{j=1}^{n-1} \theta_j(f_{cg(s_{j+1})}-f_{cg(s_{j})})\rb^\aa m(dx). \ee
\end{lem}

\begin{proof}
By the definition of representations in terms of stable integrals, we have for all $c>0$ and $g\in\GG$
\bea
&&\exp\(-c^{-H\aa}\int_E \lb\sum_{j=1}^{n-1}
\theta_j(f_{cg(s_{j+1})}-f_{cg(s_{j})})\rb^\aa m(dx)\)\nn\\
&=& \E
\exp\(i\sum_{j=1}^n \th_j c^{-H}(X_{cg(s_{j+1})}-X_{cg(s_j)})\)\nn\\
&=& \E
\exp\(i\sum_{j=1}^n \th_j (X_{s_{j+1}}-X_{s_j})\)\nn\\
&=&\exp\(- \int_E \lb\sum_{j=1}^{n-1} \theta_j(f_{s_{j+1}}-f_{s_{j}})
\rb^\aa m(dx)\). \eea
Note that the middle equality holds since $X_s$ is an $(\aa,H)$-FSF.
\end{proof}

\begin{thm}\label{thm:sssi}
Assume the usual conditions in \eqref{uc}. Then, the
$\aa$-stabilized random-time kernel $\{f_{\t_t}\}_{t\in \R^N}$ is a representation for
an $(\aa,\H)$-FSF with $$\H=H'H$$ and control measure $\P'\times m$.
\end{thm}
\begin{proof}
Let $c>0$ and $g\in\GG(\R^N)$.
Using Lemma \ref{lemma: ss} and the fact that $\t_t$ is $H'$-ss, we have
\bea\label{maincalc}
&&\exp\(-c^{-H'H\aa}\int_E \E'\lb\sum_{j=1}^{n-1}
\theta_j(f_{\t_{cg(t_{j+1})}}-f_{\t_{cg(t_{j})}})\rb^\aa m(dx)\)\nn\\
&=&\exp\(-c^{-H'H\aa}\int_E \E'\lb\sum_{j=1}^{n-1} \theta_j(f_{c^{H'}\t_{g(t_{j+1})}}-f_{c^{H'}\t_{g(t_{j})}}) \rb^\aa m(dx)\)\nn\\
&=&\exp\(- \E'\int_E \lb\sum_{j=1}^{n-1} \theta_j(f_{\t_{g(t_{j+1})}}-f_{\t_{g(t_{j})}}) \rb^\aa m(dx)\).
\eea

Since $\t_t$ is sis, there is a random translation $h(\omega',g,\cdot)\in\GG(\R^d)$ such that
$$(\t_{g(t_1)},\ldots,\t_{g(t_n)})\ed (h(\t_{t_1}),\ldots,h(\t_{t_n})),$$ thus
\be\nn\int_E \lb \sum_{j=1}^{n-1} \theta_j(f_{\t_{g(t_{j+1})}}-f_{\t_{g(t_{j})}})\rb^\aa m(dx) \ed \int_E \lb \sum_{j=1}^{n-1} \theta_j(f_{h(\t_{t_{j+1}})}-f_{h(\t_{t_{j}})})\rb^\aa m(dx).\ee
By the above equation and Lemma \ref{lemma: ss}, the right side of \eqref{maincalc} equals
\be\label{maincalc2}
\exp\(- \int_E \E'\lb\sum_{j=1}^{n-1} \theta_j(f_{\t_{t_{j+1}}}-f_{\t_{t_{j}}}) \rb^\aa m(dx)\).
\ee

Finally, \eqref{maincalc} and \eqref{maincalc2} show that
$$\exp\(-c^{-H'H\aa}\int_E \E'\lb\sum_{j=1}^{n-1}
\theta_j(f_{\t_{cg(t_{j+1})}}-f_{\t_{cg(t_{j})}})\rb^\aa m(dx)\)$$
does not depend on $c$ nor $g$, so using Lemma \ref{lemma: ss} once again, we have that
$\{f_{\t_t}\}_{t\in\R^N}$ represents an $(\aa,H'H)$-FSF. 
\end{proof}

\section{Examples}\label{sec:examples}

For the remainder of the paper, we let $\t_t^{H'}$ be a fractional Brownian field with parameter $H'$. 
This keeps us from getting bogged down in unilluminating details, while at the same time allows us to illustrate the properties and effects of $\aa$-stabilized random-time kernels.

\begin{ex}[$\aa$-stable Levy motion]\label{lt example}
Suppose $N=d=1$, $f_s=1_{[0,s]}$,
and $\t_t^{H'}$ is a fractional Brownian motion. If $M_\aa$ is an $\aa$-stable random measure
on $\Om'\times\R$ with control measure $\P'\times m$, then the random-time $\aa$-stable Levy motion
\bea\nn
\int_{\Om'\times\R} f_{\t_t^{H'}}(x) \; M_\aa(d\omega', dx)
&=& \int_{\Om'\times\R} 1_{[0,\t_t^{H'}(\omega')]}(x) \; M_\aa(d\omega', dx),\ t\ge 0,\label{ifsm}
\eea
is equivalent to what is known as an Indicator FSM with Hurst parameter $\H=H'/\aa$.  If $\t_t^{H'}>0$, we have the
following interpretation:
$[0,\t_t^{H'}(\omega')]:=[\t_t^{H'}(\omega'),0]$.
\end{ex}

\begin{ex}[Levy-Chentsov fields]\label{chentsov example}
Let us extend the above example by letting $N\ge 1, d\ge 1$, and
$f_s=1_{B(s/2)}$
where $B(s/2)$ is the ball in $\R^d$ centered at $s/2$ with radius $\|s/2\|$ (here $\|\cdot\|$ is the Euclidean norm).
Suppose also that points in $\R^d$ are written $(\phi,r)\in S_n\times\R_+$ where $S_n$ is the $(n-1)$-dimensional unit sphere.
We can identify points in $\R^d$ with hyperplanes of codimension $1$ which are distance $r$ from the origin and which are orthogonal to $\phi$.
The ball $B(s/2)$ can be thought of as the set of hyperplanes which pass between the origin and $s$.
 
In \cite{chentsov}, it was shown that $\{f_s\}_{s\in\R^d}$ is a representation of a $1/\aa$-FSF where $M_\aa$ has control measure
$m(C^{-1}d\phi,dr)$. Here $C^{-1}d\phi$ is a constant multiple of Lebesgue measure on $S_n$, scaled so that the ball
corresponding to the unit time $e_1$ has
measure one:
\bea
C&=&\frac{1}{2}\int_{S_n} |(\phi\cdot e_1)|\, d\phi\nn\\
&=& \int_{S_n} \text{Leb}\{r:0<r<(\phi\cdot e_1)\}\, d\phi.
\eea

Note that the increments $X_{s_2}-X_{s_1}$ and
$X_{s_4}-X_{s_3}$ are independent if and only if the line segments 
$[s_1,s_2]$ and $[s_3,s_4]$ lie on the same line and do not intersect.
Letting $\t_t^{H'}$ be an $(N,d)$-FBF, we have that any two nontrivial increments of a random-time Levy-Chentsov field, $X_{\t^{H'}_{t_2}}-X_{\t^{H'}_{t_1}}$ and
$X_{\t^{H'}_{t_4}}-X_{\t^{H'}_{t_3}}$, are dependent since $[\t^{H'}_{t_1},\t^{H'}_{t_2}]$ and $[\t^{H'}_{t_3},\t^{H'}_{t_4}]$ are colinear with zero probability when $d>1$
and intersect with positive probability when $d=1$.

Finally, one may check that when $N=d=1$, these fields are reduced to Example \ref{lt example}.

\end{ex}

\begin{ex}[Random-slope FSMs]\label{slope example}

Let $d=1$. The so-called random-slope FSM (see \cite{kono1991self}) is given by the integral representation
\be\nn
\int_{[0,1]} s \,M_\aa(dx),
\ee
where $M_\aa$ is a stable random measure on $[0,1]$ with Lebesgue control measure. The process is almost surely a
line where the slope is given by the random variable $S_\aa=\int_{[0,1]}\,M_\aa(dx)$. Since lines are $1$-sssi, these processes 
can be considered degenerate FSMs.

If we replace
$s$ with an FBF $\t^{H'}_t$ supported on $\Omega'$, then a
random-time random-slope FSM,
\be\label{ex:subGaussian} \int_{\Omega'} \t^{H'}_t \,M_\aa(d\omega')\ed \int_{\Omega'\times[0,1]} \t^{H'}_t \,M_\aa(d\omega',dx),
\ee
is a representation
of a so-called SubGaussian FSF \cite[Sections 3.7,7.9]{samorodnitsky1994stable}. By Proposition 3.8.2 there, there is a totally skewed
stable random variable $$A\sim S_{\aa/2}(\sigma=\left[\cos({\aa\pi}/4)\right]^{2/\aa},\beta=1,\mu=0)$$ such that
\eqref{ex:subGaussian} is equal in distribution to $A^{1/2}\t_t^{H'}$.
As pointed out in the remarks following Theorem 12.4.1 in \cite{samorodnitsky1994stable}, 
using the representation $A^{1/2}\t_t^{H'}$, these fields give us examples of $(\aa,1/\aa)$-FSFs
which have the remarkable feature of having continuous paths a.s.

If we generalize $\t^{H'}_t$ to be an $(\aa',H')$-FSF, then \eqref{ex:subGaussian} is called a  Substable FSF.
The fact that Substable FSFs are true FSFs follows from
 Theorem 7.9.1 in \cite{samorodnitsky1994stable}.  Theorem \ref{thm:sssi} above can therefore be considered a generalization of this result.
\end{ex}

\begin{ex}[Moving average representations of FBFs]\label{ex:lfsm}
Although we have focused on $0<\aa<2$, note that $\aa$-stabilized random-time kernels can also be used to construct
fractional Gaussian fields.
Let $M_2
(dx)$ be a Gaussian random measure on $\R^d$ with Lebesgue control measure and let
$$f_s=c_d(\|s-x\|^{(H-d)/2} - \|x\|^{(H-d)/2})$$
where $c_d$ is chosen so that $\|f_{e_1}\|_2=1$ and $e_1$ is some unit vector in $\R^d$.

Noting that Theorem \ref{thm:sssi} holds for $\aa=2$ and using the fact that FBFs are the only $H$-sssis Gaussian fields, we have that
$$\int_{\Om'\times\R^d} c_d(\|\t^{H'}_t-x\|^{(H-d)/2} - \|x\|^{(H-d)/2}) M_2(d\omega',dx)$$
is a representation of an FBF with Hurst parameter $\tilde H = H'H$. Since $0<H'<1$ we see that 
the new FBF (represented by the $\aa$-stabilized random-time kernel) has a strictly
smaller Hurst parameter then the original FBF, i.e. $\tilde H\in(0,H)$.
\end{ex}

\section{Random-time linear fractional stable motions}\label{sec:erg}

Suppose now that $d=1$. Recently, there have been some partial classifications of FSMs with $0<\aa<2$ using the invariance of certain ergodic-theoretic properties related to spectral representations of \sas\ processes and their associated flows \cite{rosinski1995structure,samorodnitsky2005null}.
The partial classifications use flows associated to the increment process $(X_{n+1}-X_n)_{n\in\Z}$ of an FSM $\(X_s\)_{s\in\R}$.
The increment process is stationary, thus its associated flows (and representations) can be classified as either dissipative, null-conservative, or positive-conservative.  We have used the plural `flows' since a \sas\ process does not have a unique flow;
however the property of being dissipative, null-conservative, or positive-conservative is invariant
across all flows associated to a given stationary \sas\ process.
\begin{defi}
A measurable family of functions $\{\phi_t\}_{t\in T}$ mapping $E$
onto itself and such that
\begin{enumerate}
\item
$\phi_{t+s}(x) =\phi_t(\phi_s(x))$ for all $t,s\in T$ and $x\in
E$,
\item $\phi_0(x)=x$ for all  $x\in E$
\item $m\circ \phi_t^{-1}\sim m$
for all $t\in T$
\end{enumerate}
is called a \emph{nonsingular flow}.
A measurable family $\{a_t\}_{t\in T}$ is called a \emph{cocycle} for the flow $\{\phi_t\}_{t\in T}$
if for
every $s,t\in T$ we have
\be
a_{t+s}(x) = a_s(x)a_t(\phi_s(x))\ m\text{-a.e.}.
\ee
\end{defi}

We briefly recount some results about dissipative, null-conservative, and positive-conservative representations.  For more details, we refer the reader
to the original works \cite{rosinski1995structure,samorodnitsky2005null} or to the brief review in Sections 3 and 4 of \cite{jung2010indicator}.
In \cite{rosinski1995structure}, it was
shown that measurable stationary
\sas\ processes always have spectral representation of the form
\begin{equation} \label{e:station.kernel}
f_n(x) = a_n(x) \,\left( \frac{dm\circ
\phi_n}{dm}(x)\right)^{1/\aa}  f_0\circ \phi_n(x)
\end{equation}
where $f_0\in L^\aa(E,m)$,  $\{\phi_n\}_{n\in \Z}$ is a nonsingular flow, and $\{a_n\}_{n\in \Z}$ is a cocycle, for $\{\phi_n\}_{n\in \Z}$, which takes values in $\{-1,1\}$. One may always assume the following full support condition:
\be\label{cond:support}
\text{supp}\{f_t:t\in T\}=E.
\ee
If $Y_n$ has a representation of the form \eqref{e:station.kernel}, we say that $Y_n$ is generated by $\phi_n$. Here $Y_n$ should be thought of
as an increment process $Y_n:= X_{n+1}-X_n$ of an FSM.

The usefulness of \eqref{e:station.kernel} is found in the realization that ergodic-theoretic
properties of a generating
flow $\phi_n$ can be related to the probabilistic properties of the \sas\ process
$Y_n$.  In particular,
certain
ergodic-theoretic properties of the flow are found to be invariant from representation
to representation.
In Theorem 4.1 of \cite{rosinski1995structure} it was shown that the dissipative-conservative decomposition of a flow is
one such representation-invariant property. 
The following result appeared as Corollary 4.2 in \cite{rosinski1995structure} and has been adapted to the current context:
\begin{thm}[Rosinski]\label{thm:rosinski}
Suppose $0<\aa<2$. A stationary \sas\ process is generated by a conservative (dissipative, respectively) flow if and only if for some (all) measurable spectral representation
$\{f_n\}_{n\in\Z}\subset L^{\aa}(E,m)$ satisfying \eqref{cond:support}, the sum
\be
\sum_{n\in \Z} |f_n(x)|^{\aa}
\ee
is infinite (finite) $m$-a.e. on $E$.
\end{thm}

In \cite{samorodnitsky2005null}, another representation-invariant property of flows, the {\it positive-null} decomposition of stationary
\sas\ processes, was introduced. Dissipative flows are always null, whereas conservative flows can be either null or positive. Perhaps the best way
to think about null \sas\ processes is given in the following result.

\begin{thm}[Samorodnitsky]\label{thm:sam}
Suppose $0<\aa<2$. A stationary \sas\ process is generated by a null flow if and only if is is ergodic.
\end{thm}

We will shortly see that the increment process of a
random-time L-FSM is mixing which implies that their flows are either dissipative or conservative null.  In order to show this, we will need a result which appeared as Theorem 2.7 of \cite{gross1994some}:

\begin{lem}[A. Gross]\label{lem:gross}
Suppose $Y_n$ is a stationary \sas\ process, and assume $\{f_n\}\subset L^{\aa}(E,m)$ is a representation of $Y_n$.
Then, $Y_n$ is mixing if and only if for every compact $K\subset\R-\{0\}$ and
every $\eps>0$,
\be
\lim_{n\to \ff} m\{x:f_0\in K, |f_n|>\eps\}=0.
\ee
\end{lem}

Let us continue to assume in this section, that $\t_t^{H'}$ is a fractional Brownian motion with parameter $H'$. 
Let $d=N=1$ in Example \ref{ex:lfsm} and modify the kernels to
 \bea
f_{\aa,H}(a,b;s,x):= a\((s-x)_+^{H-1/\aa} - (-x)_+^{H-1/\aa}\) +
 b\((s-x)_-^{H-1/\aa} - (-x)_-^{H-1/\aa}\)\nn
 \eea
so that in particular
\bea \label{rtlfsm kernel}
&&f_{\aa,H}^{\t^{H'}}(a,b;t,x;\omega'):=\\&&\nn
a\((\t_t^{H'}(\omega')-x)_+^{H-1/\aa} - (-x)_+^{H-1/\aa}\) +
b\((\t_t^{H'}(\omega')-x)_-^{H-1/\aa} - (-x)_-^{H-1/\aa}\). \eea  
Additionally, we normalize the scale parameters at time $s=1$ by choosing $a,b\ge 0$ so that
\be\label{normalized}
\|f_{\aa,H}(a,b;1,x)\|_\aa=1.
\ee

We say that an FSM is dissipative, null-conservative, or positive conservative if its increment process is dissipative, null-conservative, or positive conservative.


\begin{prop}\label{prop:lfsm}
Random-time L-FSMs given by the integral representations \be\label{OneOverAlpha}
\int_{\Om'\times\R} f_{\aa,H}^{\t^{H'}}(a,b;t,x;\omega')\,M_\aa(d\omega', dx),\ t\ge
0\ee 
are null-conservative. 
\end{prop}

\begin{proof}
Fix $x\in\R$, $\eps>0$.  Also, assume without loss of generality that $b=0,a>0$. Let $X^\t_t$ be the Random-time L-FSM
represented by $f_{\aa,H}^{\t^{H'}}(a,b;t,x;\omega')$.

Since $\t^{H'}_t$ is an FBM, we have that a.s.
$$\t^{H'}_n<x\text{ and }\t^{H'}_{n+1}>x+\eps$$
infinitely often. Thus for almost every $(x,\omega')$ we have
\be\lb f_{\aa,H}^{\t^{H'}}(a,b;n+1,x;\omega')-f_{\aa,H}^{\t^{H'}}(a,b;n,x;\omega')\rb\ge a\eps^{H-1/\aa}\ee
infinitely often, so by Theorem \ref{thm:rosinski} $X^\t_t$ is conservative.

We next show that the increments of $X^\t_t$ are mixing which, by Theorem \ref{thm:sam}, will imply that
$X^\t_t$ is null-conservative.
Fix a compact $K\subset\R-\{0\}$ and let $\delta=\dist(K,0)$.
Let $$B_\delta^n(\omega')=\lk x:\lb a\((\tH_{n}-x)_+^{H-1/\aa}-(-x)_+^{H-1/\aa}\)\rb>\delta\rk.$$
By Lemma \ref{lem:gross}, it suffices to show that
\be
\lim_{n\to\ff}(\P'\times\text{Leb.})\{(\omega',x):x\in B^1_\delta\cap B^n_\delta\}=0.
\ee

Choose $M$ large enough so that
\be\label{eqn:tail}
\int_M^\ff\P'(\tH_1>x)\,dx<1/M^3,
\ee
and choose $C$ so that if $|\tH_1|\le L$ then $B_\delta^1\subset (-CL^2,L]$ for all $L>0$.

We have that
\bea \nn
&&(\P'\times\text{Leb.})\{(\omega',x):x\in B^1_\delta\cap B^n_\delta\}\\
&=&(\P'\times\text{Leb.})\{(\omega',x):|\tH_1|>M, x\in B^1_\delta\cap B^n_\delta\}\nn\\
&&+(\P'\times\text{Leb.})\{(\omega',x):|\tH_1|\le M, x\in B^1_\delta\cap B^n_\delta\}\nn\\
&\le& 2\int_{M}^{\infty} \P'(\t^{H'}_1>x)\, dx + (\P'\times\text{Leb.})\{(\omega',x):|\tH_1|\le M, x\in B^1_\delta\cap B^n_\delta\} \nn\\
&\le& 2M^{-3} + (CM^2+M)\sup_{x\in(-CM^2,M]}\P'\{\omega':x\in B^n_\delta\}
\eea
Since the right side above can be made arbitrarily small by choosing $M$ and then $n$ appropriately, the result is proved.
\end{proof}

\section{Open questions}
\begin{enumerate}
\item
This work has generalized Indicator FSMs which are, in a sense, dual
to Local-time FSMs.  One can ask whether or not the
generalization presented in this work extends somehow to Local-time FSMs. We outline a possible extension in the special
case of L-FSMs.

In \cite[Section 6]{geman1980occupation}, the local time $\ell_\pi$
of an $(N,d)$-field $(\t_t)_{t\in\R^N}$ with respect to a measure
$\pi$, different from Lebesgue, is described. If it exists, then
 $$
 \pi(\{t\in A:\t_t\in B\}) = \int_B \ell_\pi(x,A)\,dx\quad
 A\in\BB(\R^N),\, B\in\BB(\R^d).
 $$

Suppose $\pi$ is absolutely continuous with respect to Lebesgue
measure $\lambda$ on $\R^N$. If $\t_t$ is a locally nondeterministic
$(N,d,\aa)$-field, then $\ell_\pi$ exists and is jointly continuous
in space and time (see \cite{xiao2011properties} for a review of local nondeterminism in the $\aa$-stable setting).

Fix $-\frac{1}{\aa}<H<0$ and consider the family of Radon-Nikodym
derivatives
$$\(\frac{d\pi_s}{d\lambda}(t)\)_{s\in\R^N}=\(\|s-t\|^{H} - \|-t\|^{H}\)_{s\in\R^N}.$$ 
By the occupation time formula $\ell_{\lambda}(x,\{t:t\in\R^N\})$ is clearly not $L^1$ even for a transient $\t_t$.
However, the measures $\{\pi_s\}$ are finite, thus if one can show
that $\ell_{\pi_s}(x,\{t:t\in\R^N\})$ is in $L^\aa(\Omega\times\R^N)$ for each
$s$, then following the procedures of \cite{CS},
$\(\ell_{\pi_s}(x,\R^N)\)_{s\in\R^N}$ is a representation of an
isotropic FSF.

If one is familiar with the theory of Gaussian processes, then one
might notice that the above scheme is similar to generalizing
Gaussian random measures to Isonormal Gaussian processes.  Such a
generalization is not simply a theoretical construct, but can be
practical. For example, if one is concerned about how much time a
stock-market-related process spends at a given point $x$ but daytime hours are more
important than nighttime hours, then a weighted measure $\pi(dt)$
allows one to express this within the local time.

\item In Section \ref{sec:erg} we found a family of null-conservative FSMs
for each $(\aa,H)$ with $0<H<\max(1,1/\aa)$.  In
\cite{CS,jung2010indicator}, null-conservative FSMs  were also found
for each $(\aa,H)$ with $0<H<\max(1,1/\aa)$.  A natural question one
may ask is, for a given pair $(\aa,H)$, are the processes in
Section \ref{sec:erg} different (in the sense of finite dimensional
distributions up to constant multiples) from those of \cite{CS} or
those of \cite{jung2010indicator}? Moreover, for a given pair $(a,b)$ satisfying \eqref{normalized},
are the Random-time L-FSMs different?  Such a result would be analogous to Theorem 7.4.5 in \cite{samorodnitsky1994stable}.
\end{enumerate}


\end{document}